\documentclass{article}
\usepackage{amsmath, tikz-cd}
\usepackage{multicol}
\usepackage{subfigure}
\usepackage{amssymb}
\usepackage{enumitem}
\usepackage{amsthm}

\begin{document}
\title{Linear Fractional Self-Maps of the Unit Ball in $\mathbb{C}^N$.}
 
\author{Michael R. Pilla\\ Ball State University\\michael.pilla@bsu.edu}

\maketitle

\bibliographystyle{amsplain}

\numberwithin{equation}{section}
\theoremstyle{plain}
\newtheorem{Proposition}[equation]{Proposition}
\newtheorem{Corollary}[equation]{Corollary}
\newtheorem*{Corollary*}{Corollary}
\newtheorem{Theorem}[equation]{Theorem}
\newtheorem*{Theorem*}{Theorem}
\newtheorem{Lemma}[equation]{Lemma}
\theoremstyle{definition}
\newtheorem{Definition}[equation]{Definition}
\newtheorem{Conjecture}[equation]{Conjecture}
\newtheorem{Example}[equation]{Example}
\newtheorem{Exercise}[equation]{Exercise}
\newtheorem{Remark}[equation]{Remark}
\newtheorem{Question}[equation]{Question}

\setlength{\parskip}{12pt}
\setlength{\parindent}{0in}

\newcommand{\R}{\mathbb{R}}
\newcommand{\C}{\mathbb{C}}
\newcommand{\D}{\mathbb{D}}

\begin{abstract}
Determining the range of complex maps plays a fundamental role in the study of several complex variables and operator theory. In particular, one is often interested in determining when a given holomorphic function is a self-map of the unit ball. In this paper, we discuss a class of maps in $\mathbb{C}^N$ that generalize linear fractional maps. We then proceed to determine precisely when such a map is a self-map of the unit ball. In particular, we take a novel approach obtaining numerous new results about this class of maps along the way.
\end{abstract}

\section{Introduction.}

Our goal is to completely characterize which linear fractional transformations map the ball $\mathbb{B}^N$ into itself for $N>1$. We recall the unit ball is defined as
$$\mathbb{B}^N=\{z \in \mathbb{C}^N \mid |z|<1\}.$$
A characterization for $N=1$ can be determined by noting that linear fractional maps (LFMs) in the disk map the boundary of the disk to a dilated and translated automorphism of the disk. This determination makes critical use of the well-known fact that LFMs map generalized circles to generalized circles, where a generalized circle is defined to include lines. This unique feature of LFMs will prevent our results in this section from generalizing to rational maps of degree greater than $1$. 

In order to replicate this line of reasoning in higher dimensions, we first produce some preliminary results about LFMs.

\section{LFMs in $\mathbb{C}^N$}
Recall that a linear fractional map in $\mathbb{C}$ is defined as 

$$\phi(z)=\frac{az+b}{cz+d}.$$

where the coefficients $a$, $b$, $c$, and $d$ are complex numbers such that $ad-bc \neq 0$ (otherwise $\phi(z)$ is a constant function). 

Given such a map $\phi$, we will find it insightful to make use of its associated matrix, defined as 

\begin{equation*}
    m_{\phi} = \begin{pmatrix}
        a & b \\
        c & d
     \end{pmatrix}.
\end{equation*}

For $c \neq0$, we must have $z \neq -\frac{d}{c}$, from which it follows, in order to avoid poles in the disk, that the inequality 

$$|d|^2-|c|^2>0$$

must hold. While this is a necessary condition for a LFM to be a self-map of the disk, it is not sufficient. It turns out a stronger condition holds.

\begin{Theorem}
The LFM $\phi(z)=\frac{az+b}{cz+d}$ is a self-map of the disk if and only if the following inequality holds:
$$ |b \overline{d}-a\overline{c}|+|ad-bc|\leq |d|^2-|c|^2.$$
\end{Theorem}

It appears the most accessible location of the proof, which is relatively simple, can be found in Mart\'in's paper \cite{Martin}. Another location of the proof can be found in \cite{Lipa}. Other than this, it seem to have primarily resided in the mathematical folklore.

We aim to obtain a parallel result in $\mathbb{B}^N$ for $N>1$. In order to discuss linear fractional self-maps in higher dimensions, we must first state what it means to be a linear fractional map in $\mathbb{C}^N$. One remarkable fact about LFMs in $\mathbb{C}$ is that they are exactly the linear transformations in homogeneous coordinates. Taking the perspective that LFMs in $\mathbb{C}^N$ should share this property, we arrive at the following definition. For more details, see \cite{Pilla}.

\begin{Definition}
We say $\phi$ is a linear fractional map in $\mathbb{C}^N$ if 

\begin{equation}
\label{LFM}
\phi(z)=\frac{Az+B}{\langle z,C\rangle+D}.
\end{equation}

\noindent where $A$ is an $N \times N$ matrix, $B$ and $C$ are column vectors in $\mathbb{C}^N$, $D \in \mathbb{C}$, and $\langle \cdot, \cdot \rangle$ is the standard inner product.
\end{Definition}

As in the case of the disk, the domain of $\phi$ is given by $\{z \in \mathbb{C}^N \mid \langle z, C \rangle+D \neq 0\}$. In order to avoid poles in the unit ball, we then require $|D|^2>|C|^2$ since $z=-\frac{DC}{|C|^2}$ is a zero of $\langle z, C \rangle +D$ which for our case requires  $\left|-\frac{DC}{|C|^2}\right|>1$.

We define the associated matrix $m_{\phi}$ of the linear fractional map $\phi(z)=\frac{Az+B}{\langle z,C\rangle+D}$ to be given by

\begin{equation*}
    m_{\phi} = \begin{pmatrix}
        A & B \\
        C^* & D
     \end{pmatrix}
\end{equation*}

Note that for a LFM in $\mathbb{C}^N$, $m_{\phi}$ will be an $(n+1) \times (n+1)$ matrix. For LFMs $\phi$ and $\psi$, a routine calculation shows that $m_{\phi \circ \psi}=m_{\phi}m_{\psi}$ and  $m_{\phi^{-1}}=(m_{\phi})^{-1}$. Thus, function composition corresponds to matrix multiplication and $m_{\phi}$ is invertible as a matrix if and only if the LFM $\phi$ has an inverse. In order to ensure this, we will presume $\phi$ is one-to-one.

In order to generalize the results of Mart\'in, we ask if our LFMs map generalized spheres to generalized spheres in $\mathbb{C}^N$ for $N>1$. It turns out this is just a little too much to ask. We do, however, have the following.

\begin{Theorem}
Linear fractional maps in $\mathbb{C}^N$ map generalized ellipsoids to generalized ellipsoids.
\end{Theorem}

For LFMs defined on the closed unit ball, this results was proved by Cowen and MacCluer \cite{Cowen}. For this result, a generalized ellipsoid is defined as the image of $\mathbb{B}^N$ under an invertible linear tranformation composed with a translation.

We aim to discuss these objects in a more tractable way, demonstrating a novel proof of the ellipsoid-preserving properties of LFMs, more reminiscent of the proof in one complex variable. In order to do so, however, we first investigate how to decompose LFMs into their atomic parts.

\section{Decomposition of Linear Fractional Maps}

Our goal in this section will be to decompose linear fractional maps into a canon of simpler maps. In one variable, one obtains a rather nice decomposition through some clever rewriting. The simplest way to generalize this result is to appeal to decomposition of the associated matrix. Given an invertible matrix $m_{\phi}$ associated with a LFM $\phi$, Bruhat decompositon tells us we may write $m_{\phi}$ as a (matrix) product of the following: permutation matrices, upper triangular unipotent matrices, and diagonal matrices. Here, unipotent means each diagonal entry is $1$. See, for example, \cite{humphreys} for more details.

 Permutation matrices may be further decomposed into a product of matrices that permute two of the variables and leave the others fixed. The associated LFMs play the analogue of inversion maps in one variable. In hindsight, perhaps this is what ought to be expected. If one takes the perspective that, in the complex plane, inversion is simply reflection about the boundary of the disk, then inversion analogues in higher dimensions manifesting themselves as reflections permuting two variables is not so surprising. This motivates the following definition.

\begin{Definition}
A linear fractional map from $\mathbb{C}^N$ to $\mathbb{C}^N$ will be called a \textbf{reflection} if its associated matrix is a permutation matrix that permutes precisely two variables and fixes the rest.
\end{Definition}

These inversion analogues also draw attention to some of the stark differences in higher dimensions. In $\mathbb{C}$, the map $\frac{1}{z}$ possesses several useful properties, including the fact that it inverts the unit disk while mapping its boundary onto itself. The author shouldn't need to do much convincing to remind the reader of the countless results in complex analysis that utilize this fact. Thus, it is worth understanding the nature of its analogues in several variables. 

This also draws attention to the fact that we are making a conscious decision to first investigate self-maps of the unit ball. The study of one complex variable bifurcates into the study of $\mathbb{C}$ and $\mathbb{D}$ since all other simply-connected open subsets of $\mathbb{C}$ are biholomorphically equivalent to $\mathbb{D}$ by the Riemann Mapping Theorem. It's well-known that this theorem fails spectacularly even in $\mathbb{C}^2$. In fact, even the unit polydisc, defined by 
$$D^N=\{z \in \mathbb{C}^N \mid |z_j|<1 \hspace{3mm} \forall \hspace{3mm} \{j\}_1^N\}$$
is not biholomorphically equivalent to the unit ball $\mathbb{B}^N$ for $N>1$. With this in mind, it is interesting to note that our reflections do map the boundary of the polydisc into itself. One may also conclude by our decomposition that inversion of $\mathbb{B}^N$ by a linear fractional map is unique to the case $N=1$.

To handle the upper triangular matrices, we make the following definition.

\begin{Definition}
A linear fractional map from $\mathbb{C}^N$ to $\mathbb{C}^N$ will be called a \textbf{multi-linear map} if it can be written in the form $\left(\langle z, \alpha_1 \rangle +c_1,  \langle z, \alpha_2 \rangle +c_2\right)$ for some vectors $\alpha_1, \alpha_2$ and constants $c_1, c_2$.
\end{Definition}

The below theorem then quickly follows from the Bruhat decomposition of the associated matrix of the LFM. 

\begin{Theorem}
Any LFM in $\mathbb{C}^N$ can be written as a composition of multi-linear maps and reflections.
\end{Theorem}

\section{Generalized Ellipsoids}

Now that we have decomposed our LFMs, we may proceed to our desired result. It suffices to show our results for reflections and multi-linear maps. 

\begin{Theorem}
Let $\phi$ be a linear fractional map in $\mathbb{B}^N$. Then $\phi$ maps generalized ellipsoids to generalized ellipsoids.
\end{Theorem}

\begin{proof}
We first show it is true for reflections. It suffices to show that the map $$\left(\frac{z_1}{z_N}, \frac{z_2}{z_N},\dots, \frac{z_{N-1}}{z_N}, \frac{1}{z_N}\right)$$ maps generalized ellipsoids to generalized ellipsoids. Recalling that a generalized ellipsoid is defined as the image of the unit ball under a linear transformation and translation, without loss of generality, it is clear to see that one may write a generalized ellipsoid in standard form as
\begin{equation*}
   \sum_{i=1}^N \alpha_i |z_i|^2+\sum_{i=1}^N\beta_i \Re(z_i)+\sum_{i=1}^N \gamma_i \Im(z_i)+\delta=0.
\end{equation*}

 Noting that $\frac{1}{z_N}=\frac{\overline{z_N}}{|z_N|^2}$, we have 
 
\begin{align*}
&=\sum_{i=1}^{N-1} \alpha_i \left|\frac{z_i}{z_N}\right|^2+\alpha_N\left|\frac{1}{z_N}\right|^2+\sum_{i=1}^{N-1}\beta_i \Re\left(\frac{z_i\overline{z_N}}{|z_N|^2}\right)+\beta_N \Re\left(\frac{\overline{z_N}}{|z_N|^2}\right)\\
&+\sum_{i=1}^{N-1}\gamma_i \Im\left(\frac{z_i\overline{z_N}}{|z_N|^2}\right)+\gamma_N \Im\left(\frac{\overline{z_N}}{|z_N|^2}\right)+\delta=0.
\end{align*}

which simplifies to 

\begin{align*}
&\sum_{i=1}^{N-1} \alpha_i \left|z_i\right|^2+\alpha_N+\sum_{i=1}^{N-1}\beta_i \Re\left(z_i\overline{z_N}\right)+\beta_N \Re\left(\overline{z_N}\right)\\
&+\sum_{i=1}^{N-1}\gamma_i \Im\left(z_i\overline{z_N}\right)+\gamma_N \Im\left(\overline{z_N}\right)+\delta|z_N|^2=0.
\end{align*}

which simplifies to 

\begin{align*}
&=\sum_{i=1}^{N-1} \alpha_i \left|z_i\right|^2+\alpha_N+\sum_{i=1}^{N-1}\beta_i \Re(z_i)\Re(z_N)+\sum_{i=1}^{N-1}\beta_i \Im(z_i)\Im(z_N)\\
&+\beta_N \Re\left(z_N\right)+\sum_{i=1}^{N-1}\gamma_i \Re(z_i)\Im(z_N)-\sum_{i=1}^{N-1}\gamma_i \Im(z_i)\Re(z_N)\\
&-\gamma_N \Im\left(z_N\right)+\delta|z_N|^2=0.
\end{align*}

We next note that this is an equation of a generalized ellipsoid in nonstandard form.

Next we consider multi-linear maps. Letting $z_k=x_k+iy_k$, $\alpha_{ij}=u_{jk}+iv_{jk}$, and $\beta_{j}=s_j+it_j$, we have 

\begin{align*}
&\sum_{j=1}^N \left|\beta_i+\sum_{k=1}^N\alpha_{jk}z_k\right|^2=\sum_{j=1}^N \left|s_j+\sum_{k=1}^N \left(u_{jk}x_k-v_{jk}y_k\right)+i\left(t_j+\sum_{k=1}^N(v_{jk}x_k+u_{jk}y_k)\right)\right|^2\\
&=\sum_{j=1}^N\left[\left(s_j+\sum_{k=1}^N \left(u_{jk}x_k-v_{jk}y_k\right)\right)^2+\left(t_j+\sum_{k=1}^N(v_{jk}x_k+u_{jk}y_k)\right)^2\right]\\
\end{align*}

From which it is evident, after expansion, that we will have the form of a generalized ellipsoid in nonstandard form.

\end{proof}

\section{Automorphisms of $\mathbb{B}^N$ and $\mathbb{D}^N$.}

Recall that for $\alpha$ with $|\alpha|<1$, there is an an automorphism of the disk $\phi_{\alpha}$ given by

\begin{equation}
\phi_{\alpha}(z)=\frac{\alpha-z}{1-\overline{\alpha}z}
\end{equation}

such that $\phi_{\alpha}(0)=\alpha$ and $\phi_{\alpha}(\alpha)=0$.

In fact, all automorphism of the disk can be written as 
\begin{equation}\label{auto}
\phi_{\alpha}(z)=e^{i \theta}\frac{\alpha-z}{1-\overline{\alpha}z}
\end{equation}

for some $\alpha$ in the disk. In particular, note that up to rotation, all automorphisms of the disk are LFMs. Furthermore, for $N=1$, the polydisk and unit ball are equivalent. 

One might hope for similar results to be true for $\mathbb{B}^N$ and $\mathbb{D}^N$ when $N>1$. It turns out such hopes are realized, although the group of automorphisms are different for each.

In the polydisk $\mathbb{D}^N$, the automorphisms consist of automorphisms of the disk such as those in Equation \ref{auto}, composed with permutations of the disks. See Rudin \cite{Rudin}. Since the polydisk is already a Cartesian product of disks in the complex plane, this result is not so surprising.

Let $\langle \phantom{z} , \phantom{z} \rangle$ denote the standard inner product and for $\alpha \in \mathbb{B}^N$ let $P_{\alpha}$ denote the orthogonal projection of $\mathbb{C}^N$ onto the span of $\alpha$ and $Q_{\alpha}(z)=z-P_{\alpha}(z)$ denote the projection onto the orthogonal complement of the span of $\alpha$. Letting $P_0(z)=0$, we have

\begin{equation}P_{\alpha}(z)=\frac{\langle z, \alpha \rangle}{\langle \alpha, \alpha \rangle}\alpha , \hspace{5mm} \alpha \neq 0.
\end{equation}

Next let $s_{\alpha}=\sqrt{1-|\alpha|^2}$. We define the following map:

\begin{equation}\label{alpha}
\phi_{\alpha}(z)=\frac{\alpha-P_{\alpha}(z)-s_{\alpha}Q_{\alpha}(z)}{1-\langle z, \alpha \rangle}
\end{equation}

for $\alpha \in \mathbb{B}^N$. Note that this map is an involution and exchanges the points $0$ and $\alpha$. 

One may show that all automorphisms of $\mathbb{B}^N$ are of the form $U\phi_{\alpha}$ where $U$ is an $n \times n$ unitary matrix (see Rudin \cite{Rudin1}). As in the disk, $U$ can be seen as a rotation of the ball.

Although it would take us too far afield to discuss domains beyond the unit ball and polydisk, for completion we state the following result due to Ahn, Byun, and Park \cite{Ahn}.  It turns out, for Hartog type domains over classical Hermitian symmetric spaces, the automorphisms consists simply of the classical Lie groups.

\section{Linear Fractional Self-Maps of $\mathbb{B}^N$}

Given a LFM $\phi(z)=\frac{Az+B}{\langle z, C \rangle +D}$, the main goal of this section is to study the range of $\phi$ on the unit ball.  While we are not the first to study this, we will be taking a novel approach. 

The first results in this direction were by Cowen and MacCluer \cite{Cowen} who showed that a LFM is a self-map of the unit ball if and only if its associated matrix is a Krein contraction. For $z_1$ in $\mathbb{C}^N$ and $z_2 \neq 0$ in $\mathbb{C}$, we identify $z=(z_1, z_2)$ with $v=\frac{z_1}{z_2}$. For $v, w$ described in this way, define the Krein inner product by $[v,w]=\langle Jv, w \rangle$ where $\langle \phantom{z}, \phantom{z} \rangle$ denotes the standard inner product with 

$$J=\begin{pmatrix}
        I & 0 \\
        0 & -1
     \end{pmatrix}.$$
Thus, a map $\phi$ is a self-map of $\mathbb{B}^N$ if and only if

$$t^2[m_{\phi}v, m_{\phi}v] \leq[v,v]$$

for some $t>0$. Without a way to determine $t$, however, this result doesn't give a practical way of determining whether a specified LFM is a self-map of the unit ball. Geometric characterizations were also given by Bisi and Bracci \cite{Bisi} but such characterizations are not so practical for determining whether a given LFM is a self-map of the ball.

Cowen and MacCluer's results were improved upon by Richman \cite{Richman} who showed how to determine $t$ in terms of eigenvalues and eigenvectors of $m_{\phi}$ for certain cases. Richman's approach depends on the fixed point behavior of $\phi$. In particular, if $\phi$ has no interior fixed point, it is known that $\phi$ has a boundary fixed point. In such a case, the positive multiple $t$ can be determined uniquely. In the case that $\phi$ has an interior fixed point, the best results one can obtain is given in terms of bounds (upper and lower) on $t$.

Our goal is determine when such a LFM is a self-map of the unit ball in the spirit of the folklore results codified by Mart\'in without resorting to fixed point behavior. We aim to give an easily verifiable criterion dependent only on the coefficients of the LFM. This will extend Richman's results to the general case of an interior fixed point and give an alternative criterion of determining when a given LFM is a self-map of the ball. Our approach will lie in the realization that the image of the unit ball under a LFM is a translated ellipsoid. After determining the translation, we recognize the ellipsoid as a perturbation of the unit ball. We thus apply some basic tools from linear algebra in order to obtain our desired result.  We began with a basic lemma about perturbations.

\begin{Lemma}\label{sb}
For $\alpha=(\alpha_1,\cdots, \alpha_N) \in \mathbb{B}^N\backslash \{0\}$ and $s=\sqrt{1-|\alpha|^2}$, let 
$\Omega_N=s\beta-sI-\beta$ with $\beta=\frac{1}{|\alpha|^2}(\alpha_i\overline{\alpha_j})$ where $I$ is the $N \times N$ identity matrix and $(\alpha_i\overline{\alpha_j})$ is the $N \times N$ matrix with $ij^{\text{th}}$ entry $\alpha_i\overline{\alpha_j}$. Then we have the following:

$$\Omega_N^{-1}=-\frac{1}{s|\alpha|^2}\left[|\alpha|^2I+(s-1)(\alpha_i\overline{\alpha_j})\right].$$

\end{Lemma}

\begin{proof}

First, we recall the Sherman-Morrison formula \cite{press} which states that for an invertible square matrix $A$ and column vectors $u,v$ such that $1+\overline{v}^TA^{-1}u \neq 0$, we have
$$\left(A+u\overline{v}^T\right)^{-1}=A^{-1}-\frac{A^{-1}u\overline{v}^TA^{-1}}{1+\overline{v}^TA^{-1}u}$$

where $u\overline{v}^T$ is the outer product of $u$ and $v$.

We realize that we may write $\Omega_N$ as the sum of an invertible matrix and outer product given by 
$$\Omega_N=-sI+(s-1)\beta=-sI+\left(\frac{\sqrt{s-1}}{|\alpha|}\right)\alpha \overline{\left[\left(\frac{\sqrt{s-1}}{|\alpha|}\right)\alpha\right]}^T$$

and we may thus apply the Sherman-Morrison formula. We obtain the following.

\begin{align*}
\Omega_N^{-1}&=\left(-sI+\left(\frac{\sqrt{s-1}}{|\alpha|}\right)\alpha \overline{\left[\left(\frac{\sqrt{s-1}}{|\alpha|}\right)\alpha\right]}^T\right)^{-1}\\
&=-\frac{1}{s}I+\frac{-\frac{(s-1)}{s^2|\alpha|^2}(\alpha_i\overline{\alpha_j})}{1-\frac{(s-1)}{s|\alpha|^2}|\alpha|^2}=-\frac{1}{s}I-\frac{(s-1)}{s|\alpha|^2}\frac{(\alpha_i\overline{\alpha_j})}{(s-(s-1))}\\
&=\frac{-|\alpha|^2I-(s-1)(\alpha_i\overline{\alpha_j})}{s|\alpha|^2}\\
&=-\frac{1}{s|\alpha|^2}\left(|\alpha|^2+(s-1)(\alpha_i\overline{\alpha_j})\right)=-\frac{1+(s-1)\beta}{s}.
\end{align*}

\end{proof}

We now have the following main result. We presume $C \neq 0$ to avoid the linear case and assume $|D|^2>|C|^2$.

\begin{Theorem}
Let $M_i$ represent the $i^{\textit{th}}$ row of the matrix $M$ and $s=\frac{\sqrt{|D|^2-|C|^2}}{|D|}$. A LFM 

\begin{equation}
\phi(z)=\frac{Az+B}{\langle z, C \rangle+D}
\end{equation}  

is a self-map of $\mathbb{B}^N$ if and only if 

\begin{align*}
&|BD-AC|^2+\left|\left(BC^*-A\overline{D}\left(s+(1-s)\frac{CC^*}{|C|^2}\right)\right)_i\right|^2\\
&-2\biggl< BD-AC, \left(BC^*-A\overline{D}\left(s+(1-s)\frac{CC^*}{|C|^2}\right)\right)_i \biggr> \leq (|D|^2-|C|^2)^2
\end{align*}

for all $1 \leq i \leq N$.

\end{Theorem}

\begin{proof}
If $\phi$ is a self-map of $\mathbb{B}^N$, then we saw $\mathbb{B}^N$ will map onto an ellipsoid inside of $\mathbb{B}^N$. Hence $\phi$ can be written as

\begin{equation}
\phi=RU\phi_{\alpha}+M.
\end{equation}

where $U$ is a $N \times N$ unitary matrix, $R$ is a compression matrix, $\phi_{\alpha}$ is equal to equation \ref{alpha}, and $M=(m_1,\cdots, m_N)^T$ is the center of the ellipsoid given by $\phi(\mathbb{B}^N)$. 

Next, we let $\beta=\frac{1}{|\alpha|^2}(\alpha_i \overline{\alpha_j})$ where $\alpha_i \overline{\alpha_j}$ represents the ${ij}^{\text{th}}$ entry of an $N \times N$ matrix, we note that $\phi_{\alpha}(z)=\frac{\alpha-(s+\beta-s\beta)z}{1-\alpha^*z}$. A computation shows that we must have

\begin{equation}\label{map}
\frac{Az+B}{C^*z+D}=\frac{RU(s\beta-s-\beta)z-\Gamma' z+M+RU\alpha}{1-\alpha^*z}.
\end{equation}

where $\Gamma'=(m_i\overline{\alpha_j})$.

It follows that we must have $C=-\alpha$ and $D=1$. Thus $|\alpha|^2=|C|^2$. Let $\Gamma=-\Gamma'=(m_i\overline{c_j})$. Comparing the numerators of \ref{map} implies $B=M+RU\alpha$ and $A=RU(s\beta-s-\beta)-\Gamma'$. Thus, since $C=-\alpha$ we have $M=B+\left(A-\Gamma\right)(s\beta-s-\beta)^{-1}C$ which we may rearrange as

\begin{equation}\label{M}
M+\Gamma\left(-\frac{1}{s}I-\frac{s-1}{s|C|^2}(c_i\overline{c_j})\right)C=B+A\left(-\frac{1}{s}I-\frac{s-1}{s|C|^2}(c_i\overline{c_j})\right)C.
\end{equation}

Noting that $(c_i\overline{c_j})C=|C|^2C$, the left-hand side of \ref{M} gives us

$$M-\frac{1}{s}\Gamma C-\frac{s-1}{s|C|^2}\Gamma (c_i\overline{c_j})C=M-\left(\frac{1}{s}+\frac{s-1}{s}\right)\Gamma C=M-\Gamma C=\left(1-|C|^2\right)M.$$

Likewise, after some simplification, we see that the right-hand side of \ref{M} gives $B-AC$. Thus we have

$$\left(1-|C|^2\right)M= B-AC \Rightarrow M=\frac{B-AC}{1-|C|^2}.$$

Next, since $\Gamma=MC^*=\left(\frac{B-AC}{1-|C|^2}\right)C^*$ we have that 

\begin{align*}
RU&=\left(A-\Gamma\right)(s\beta-s-\beta)^{-1}\\
&=\left(A-\Gamma\right)\left(-\frac{1}{s|C|^2}\left[|C|^2I+(s-1)(c_i\overline{c_j})\right]\right)\\
&=-\frac{1}{s}A-\frac{(s-1)}{s|C|^2}A(c_i\overline{c_j})+\frac{(s-1)}{s|C|^2}\Gamma (c_i \overline{c_j})+\frac{1}{s}\Gamma\\
&=-\frac{1}{s}A-\frac{(s-1)}{s|C|^2}ACC^*+\Gamma.\\
&=-\frac{1}{s}A-\frac{(s-1)}{s|C|^2}ACC^*+\left(\frac{B-AC}{1-|C|^2}\right)C^*\\
&=-\frac{1}{s}A+\frac{BC^*}{1-|C|^2}-\frac{s-(1-|C|^2)}{s|C|^2(1-|C|^2)}ACC^*\\
&=\frac{BC^*-A\left(s+(1-s)\frac{CC^*}{|C|^2}\right)}{1-|C|^2}.
\end{align*}

Next, we let $RU_i$ represent the $i^{\text{th}}$ row of $RU$. We apply the law of cosines to find the length of the vector connecting the displacement vector $M$ and $RU_i$ for each $i$. In order to remain in the unit ball, this set of $i$ vectors must each be less than or equal to $1$. That is, we must have 

$$|M|^2+|RU_i|^2-2\langle M, RU_i \rangle \leq 1$$

for all $1 \leq i \leq N$.

After multiplying the inequality by $1-|C|^2$ and rearranging, we obtain the set of inequalities

\begin{align*}
&|B-AC|^2+\left|\left(BC^*-A\left(s+(1-s)\frac{CC^*}{|C|^2}\right)\right)_i\right|^2\\
&-2\biggl< B-AC, \left(BC^*-A\left(s+(1-s)\frac{CC^*}{|C|^2}\right)\right)_i \biggr>\leq (1-|C|^2)^2
\end{align*}

for all $1 \leq i \leq N$.

Finally, returning from our normalized case of $D \neq 1$, we obtain our results.
\end{proof}

The linear case follows quite simply. We have $\phi(z)=Az+b$. Let $A_i$ represent the $i^{th}$ row of the matrix $A$. By the same reasoning as above we require 

$$|B|^2+|A_i|^2-2 \langle B, A_i \rangle \leq 1$$

for all $1 \leq i \leq N$.

\begin{Example}
Let $\phi$ be the linear fractional map in two complex variables given by

$$\phi(z)=\phi(z_1,z_2)=\left( \frac{z_1+1}{-z_1+3}, \frac{2z_2}{-z_1+3} \right).$$

Identifying $\langle z, C \rangle$ with $C^*z$, we can write this as 

$$\left( \frac{z_1+1}{-z_1+3}, \frac{2z_2}{-z_1+3} \right)=\frac{\begin{pmatrix}
       1 & 0  \\
       0 & 2 
     \end{pmatrix}\begin{pmatrix}
       z_1 \\
       z_2 
     \end{pmatrix}
+\begin{pmatrix}
       1 \\
       0 
     \end{pmatrix}}{(-1, 0)^T(z_1, z_2)+3}.$$

Thus we have $A=\begin{pmatrix}
       1 & 0  \\
       0 & 2 
     \end{pmatrix}$, $B=(1, 0)^T$, $C=(-1, 0)^T$, and $D=3$. Hence $BD-AC=(4,0)^T$ and
$$BC^*-A\overline{D}\left(s+(1-s)\frac{CC^*}{|C|^2}\right)=\begin{pmatrix}
       -4 & 0  \\
       0 & 4\sqrt{2} 
     \end{pmatrix}$$
Thus we have 
\begin{align*}
&|(4, 0)|^2+|(-4,0)|^2-2\langle (4,0), (-4,0) \rangle &=64=(|D|^2-|C|^2)^2\\
&|(4, 0)|^2+|(0,4\sqrt{2})|^2-2\langle (4,0), (0,4\sqrt{2}) \rangle &=46<64=(|D|^2-|C|^2)^2\\
\end{align*}

Hence $\phi$ is a self-map of the unit ball.
\end{Example}

One final comment regarding this result. As seen in the example, equality is achieved precisely when the self-map has no interior fixed points. Thus, the derived criteria for determining whether a LFM is a self-map of the unit ball also tells us whether our map has an interior fixed point or not. If there is no interior fixed point, then it is well known that our map must have a privileged fixed point on the boundary, known as the Denjoy-Wolff point \cite{maccluer}. 

\section*{Acknowledgements}

The author would like to thank Brittney Miller and Chris Felder for fruitful conversations regarding analytic self-maps of the disk. The author would also like to thank Ralph Bremigan for recommending the use of Bruhat decomposition.

\end{document}